\UseRawInputEncoding
\documentclass[a4paper,11pt]{article}
\pdfoutput=1
\usepackage[utf8]{inputenc}
\usepackage[a4paper, total={18cm, 24cm}]{geometry}

\usepackage[square,comma,numbers,sort&compress,nonamebreak]{natbib}
\usepackage{amsthm}
\usepackage{amsmath}
\usepackage{amssymb}
\usepackage{pdfpages}

\usepackage{graphicx}
\usepackage{subfigure}
\usepackage{xcolor}

\newtheorem{thm}{Theorem}
\newtheorem{lem}{Lemma}
\newtheorem{df}{Definition}

\newtheorem{prop}{Proposition}

\begin{document}

\title{$k$-fault-tolerant graphs for $p$ disjoint complete graphs of order $c$}

\author{Sylwia Cichacz$^{1}$, Agnieszka Görlich$^{1}$, Karol Suchan$^{2,1}$\\ ~~ \\
\normalsize $^1$AGH University of Science and Technology, Krakow, Poland, cichacz@agh.edu.pl, forys@agh.edu.pl\\
\normalsize $^2$Universidad Diego Portales,  Santiago, Chile,  karol.suchan@mail.udp.cl}

\date{\today}
\maketitle

\begin{abstract}
Vertex-fault-tolerance was introduced by Hayes~\cite{Hayes1976} in 1976, and since then it has been systematically studied in different aspects. In this paper we study $k$-vertex-fault-tolerant graphs for $p$ disjoint complete graphs of order $c$, i.e., graphs in which removing any $k$ vertices leaves a graph that has $p$ disjoint complete graphs of order $c$ as a subgraph. The main contribution is to describe such graphs that have the smallest possible number of edges for $k=1$, $p \geq 1$, and $c \geq 3$. Moreover, we analyze some properties of such graphs for any value of $k$.
\end{abstract}

\section{Introduction}

Given a graph $H$\footnote{Throughout the paper we deal with simple undirected graphs. For standard terms and notations in graph theory, the reader is referred to the books of Diestel~\cite{D2017} and Brandstädt et al.~\cite{Brandstadt}.} and a positive integer $k$, a graph $G$ is called \emph{vertex $k$-fault-tolerant with respect to $H$}, denoted by  $k$-FT$(H)$, if $G-S$ contains a subgraph isomorphic to $H$ for every $S\subset V(G)$ with $|S|\leq k$. 

Vertex-fault-tolerance was introduced by Hayes~\cite{Hayes1976} in 1976 as a graph theoretic model of computer or communication networks working correctly in the presence of faults. The main motivation for the problem of constructing $k$-fault-tolerant graphs is in finding fault tolerant architectures. A graph $H$ represents the desired interconnection network and a $k$-FT$(H)$ graph $G$ allows to emulate the graph $H$ even in presence of $k$ vertex (processor) faults. 

The problem has been systematically studied with different quality measures of $k$-fault-tolerant graphs. Hayes~\cite{Hayes1976} and Ajtai et al.~\cite{AABC} considered $k$-FT$(H)$ graphs with $|V(H)|+k$ vertices and the number of edges as small as possible. A different quality measure of $k$-FT$(H)$ graphs was introduced by Ueno et al. \cite{UBHS1993}, and independently by Dudek et al.~\cite{DSZ}, where the authors were interested in $k$-FT$(H)$ graphs having as few edges as possible, disregarding the number of vertices (see also~\cite{Z2}, ~\cite{PZ}). Yet another setup was studied by Alon and Chung~\cite{AC}, Ueno and Yamada~\cite{UY} and Zhang~\cite{Zhang}. They allowed $O(k)$ spare vertices in $k$-FT$(H)$ graphs and focused on minimizing the maximum degree (giving priority to the scalability of a network). Other results on $k$-fault-tolerance can be found, for example, in \cite{CS,Favaron,Luetal2021FTclubs,ZAK2014421}.

In this paper we study the variant introduced by Hayes. He characterized  $k$-FT$(H)$ graphs of order $|V(H)|+k$ in the case where $H$ is a path, a cycle or a tree of a special type~\cite{Hayes1976}. Some results related to constructing a $k$-fault-tolerant supergraph for an arbitrary graph $H$ have also been published (e.g., see~\cite{BCH},~\cite{DH}).

We focus on $k$-FT($pK_c$) graphs, where $pK_c$ is the disjoint union of $p$ complete graphs of order $c$, for $k, p \geq 1$, $c \geq 3$. Our main contribution is to describe minimum $k$-FT($pK_c$) graphs for $k=1$ and any values of $p$ and $c$ (Theorem \ref{thm:1pc}).

The paper is organized as follows. In Section \ref{se:prelim} we provide some definitions and present basic properties of $k$-FT($pK_c$) graphs, analyze their connectivity and separators of size $k$, and present an upper bound on the size of minimum $k$-FT($pK_c$) graphs. In Section \ref{se:main}, we prove that this upper bound is tight when $k=1$ and fully characterize $1$-FT($pK_c$) graphs. Finally, we present some final comments in Section \ref{se:conc}.

\section{$k$-FT($pK_c$) graphs}\label{se:prelim}

Given a graph $G=(V,E)$ and a vertex $v$, $v \in V$, we use $N_G(v)$ to denote the \textit{neighborhood} of $v$ in $G$, i.e., the set of \textit{neighbors} of $v$: vertices $y$, $y \in V$, such that $\{v,y\} \in E$. Given a set of vertices $U$, $U \subset V$, $N_G(U)$ denotes the set of vertices in $V \setminus U$ that have a neighbor in $U$. We use $N_G[v]$ to denote the \textit{closed neighborhood} of $v$: $N_G[v] = N_G(v) \cup \{v\}$. Likewise, we have $N_G[U] = N_G(U) \cup U$. When it does not lead to confusion, we omit the subscript, writing just $N(v), N[v], N(U)$ and $N[U]$. We use $K_c$ to denote a complete graph on $c$ vertices. The vertex set of a complete graph is called a \textit{clique}.

\subsection{Basic properties}\label{ss:basic}
Let us start with the main definition.

\begin{df}\label{df:ft}
Let $k$, $p$, and $c$ be integers with $k \geq 0$, $p \geq 1$, and $c \geq 2$. Let $G=(V,E)$ be a graph. We say that $G$ is \textit{$k$-FT($pK_c$)} if $G-S$ contains the union of $p$ disjoint complete graphs $K_c$ as a subgraph, for any $S$, $S \subset V$ and $|S| \leq k$. We say that $G$ is \textit{minimal $k$-FT($pK_c$)} if no proper subgraph of $G$ is $k$-FT($pK_c$). We say that $G$ is \textit{minimum $k$-FT($pK_c$)} if $|V|=pc+k$ and there is no $k$-FT($pK_c$) graph $G'$ with $|V(G')|=|V|$ and $|E(G')|<|E|$.
\end{df}

It is easy to check that, in order to prove that $G$ is \textit{$k$-FT($pK_c$)}, it is enough to verify that $G-S$ contains the union of $p$ disjoint complete graphs $K_c$ as a subgraph, for any $S$, $S \subset V$ with $|S| = k$ (equality instead of weak inequality). For the sake of simplicity, some of the following proofs use this observation without stating it explicitly.

For $c=2$ and $G=(V,E)$ such that $|V|=2p+k$, the concept of $k$-FT($pK_c$) graphs has been widely studied under the name  \textit{$k$-factor critical}  graphs. This idea was first introduced and studied for $k = 2$ by Lov\'asz \cite{Lovasz} under the term of \textit{bicritical graph}. For $k>2$ it was introduced by Yu in 1993 \cite{Yu}, and independently by Favaron in 1996 \cite{Favaron}. For even values of $k$ and $|V|=2p+k \leq 2k-2$, it was shown that if $G=(V,E)$ is minimum $k$-FT($pK_2$), then $|E|=(k + 1)|V|/2$ \cite{zhang2012minimum}. 

In what follows, we focus on the cases with $c\geq3$. Let us start with a few simple lemmas that give some basic properties of $k$-FT($pK_c$) graphs.

\begin{lem}\label{lem:ft}
Let $k$, $p$, and $c$ be integers with $k \geq 0$, $p \geq 1$, and $c \geq 3$. Let $G=(V,E)$ be a $k$-FT($pK_c$) graph with $|V|=pc+k$. Then every vertex $x$, $x\in V$: 
\begin{enumerate}
\item\label{obs:ft:c+k} belongs to a subgraph isomorphic to $K_c$ in $G$, 
\item\label{obs:ft:d} is of degree at least $c+k-1$ in $G$,
\item\label{obs:ft:c+kh} belongs to a subgraph isomorphic to $K_c$ in $G'$, $G'=G-S$, for any $S$ with $S \subset V\setminus \{x\}$ and $|S| = k$.
\end{enumerate}
\end{lem}
\begin{proof}
Let $x$ be any vertex of $G$. Choose any $S$ with $S \subset V\setminus \{v\}$ and $|S| = k$, such that $|S \cap N_G(v)|$ is maximum over all such subsets. Notice that $S$ exists, as $|V|=pc+k$ and $p$, $c$, $k$ are positive integers. Since $G$ is $k$-FT($pK_c$), $G'=G-S$ contains $p$ disjoint subgraphs isomorphic to $K_c$. Moreover, $|G'|=pc$, thus $x$ belongs to one of them - which proves Item \ref{obs:ft:c+k}. Moreover, it implies that $d_{G'}(x) \geq c-1$. By the choice of $S$, it is easy to see that $S \subset N_G(x)$. Thus $d_G(x) \geq c+k-1$ - which proves Item \ref{obs:ft:d}. Clearly, a similar reasoning applies for any choice of $S$, with $S \subset V\setminus \{v\}$ and $|S| = k$ - which proves Item \ref{obs:ft:c+kh}.
\end{proof}

In the following, we will use interchangeably the expressions that a $0$-FT($pK_c$) graph $G'$ contains $p$ disjoint complete graphs $K_c$ and that its vertex set $V(G')$ contains $p$ disjoint cliques of size $c$.

\begin{lem}\label{lem:Kc+k-1}
Let $k$, $p$, and $c$ be integers with $k \geq 0$, $p \geq 1$, and $c \geq 3$. Let $G=(V,E)$ be a $k$-FT($pK_c$) graph with $|V|=pc+k$. Then every vertex $x$, $x \in V$, with $d(x) = c+k-1$ belongs to a subgraph of $G$ isomorphic to $K_{c+k}$. Moreover, if $W$ is a separator of size $k$ in $G$ and $A$ is a component of $G-W$ of order $c$, then $W \cup V(A)$ is a clique.
\end{lem}

\begin{proof}
Let $x$ be any vertex of $G$ with $d_G(x)=c+k-1$. For any choice of $S$ with $|S|=k$ and $S \subset N_G(x)$, by Lemma \ref{lem:ft}, $x$ belongs to a copy of $K_{c}$ in $G'$, $G'=G-S$. Since $c\geq 3$, $N_{G'}(x)$ is a clique of size at least $2$. Therefore, $N_G[x]$ is also a clique.

For the second part of the lemma, note that $V(A)$ contains a vertex of degree $c+k-1$. The conclusion follows.
\end{proof}

\subsection{Connectivity and separators}\label{ss:connectivity}

Now let us proceed with some observations on the connectivity of $k$-FT($pK_c$) graphs.

\begin{lem}\label{lem:c+k-1ec}
Let $k$, $p$, and $c$ be integers with $k \geq 0$, $p \geq 1$, and $c \geq 3$. Let $G=(V,E)$ be a $k$-FT($pK_c$) graph with $|V|=pc+k$. Then $G$ is $c+k-1$ edge-connected.
\end{lem}
\begin{proof}
Suppose that $G$ is not $c+k-1$ edge-connected. It implies that there exists a separating set of at most $c+k-2$ edges $F$, $F \subset E$. Let $A$ be a component of $G-F$. Choose any $S$, with $S \subset V$ and $|S|=k$, such that $|V(A)|-b \not\equiv 0 \pmod c$ and $|\widehat{F}| \geq k$, where $b=|S \cap V(A)|$ and $\widehat{F}=\{f \in F \mid f \cap S \neq \emptyset\}$. It is easy to check that such $S$ exists and $|F'|<c-1$, where $F' = F \setminus \widehat{F}$. Let $A' = A - S$ and $G'=G - S$. It is easy to check that, since $|F'|<c-1$, $G'$ does not contain a copy of $K_c$ intersecting both $V(A')$ and $V \setminus S \setminus V(A')$. On the other hand, since $G$ is $k$-FT($pK_c$) and $|V|=pc+k$, $G'$ contains $p$ disjoint copies of $K_c$, where every vertex of $G'$ belongs to one of them. But the order of $A'$ is not a multiple of $c$ - a contradiction.
\end{proof}

\begin{lem}\label{lem:kvc}
Let $k$, $p$, and $c$ be integers with $k \geq 0$, $p \geq 1$, and $c \geq 3$. Let $G=(V,E)$ be a $k$-FT($pK_c$) graph with $|V|=pc+k$. Then $G$ is $k$ connected.
\end{lem}

\begin{proof}
Suppose that $G$ is not $k$ connected. So there is a separator $W$ in $G$ with $|W| < k$. Moreover, we can choose $S$, with $S \subset V$, $|S|=k$, and $W \subsetneq S$, in such a way that the order of one of the components of $G'$, $G'=G-S$, is not a multiple of $c$. On the other hand, since $G$ is $k$-FT($pK_c$), $G'$ contains $p$ disjoint copies of $K_c$ - a contradiction.
\end{proof}

Let us analyze some properties of $k$-FT($pK_c$) graphs $G=(V,E)$ with $|V|=pc+k$ that contain a separator of size $k$.

\begin{lem}\label{lem:Kkvc}
Let $k$, $p$, and $c$ be integers with $k \geq 0$, $p \geq 1$, $c \geq 3$, and $k<c$. Let $G=(V,E)$ be a $k$-FT($pK_c$) graph with $|V|=pc+k$. Let $W$, with $W\subset V$ and $|W|=k$, be a separator in $G$. Let $\{A_i\}_{i=1}^{z}$ be the components of $G'$, $G'=G-W$, and let $p_i=|V(A_i)|/c$ for every $i$, $1 \leq i \leq z$. Then $\sum_{i=1}^{z}p_i=p$ and $G_i$, $G_i=G[V(A_i) \cup W]$, is $k$-FT($p_iK_c$) for every $i$, $1 \leq i \leq z$. Moreover, if $W$ is a clique and $G$ is minimum $k$-FT($pK_c$), then $G_i$ is minimum $k$-FT($p_iK_c$) for every $i$, $1 \leq i \leq z$.
\end{lem}

\begin{proof}
Since $G$ is $k$-FT($pK_c$) and $|W|=k$, $G'$ contains $p$ disjoint cliques of size $c$. Since $W$ is a separator in $G$, each of each of these cliques is included in $V(A_i)$ for some $i$, $1 \leq i \leq z$. Moreover, since $\sum_{i=1}^{z}|V(A_i)|=pc$, each $V(A_i)$, $1\leq i \leq z$, is the union of exactly $p_i$ of these cliques and $|V(A_i)|=p_ic$. So $p_i$ is an integer for every $i$, $1 \leq i \leq z$, and $\sum_{i=1}^{z}p_i=p$. 

Take any $G_i$, $1\leq i \leq z$. In order to prove that $G_i$ is $k$-FT($p_iK_c$), we need to show that $G_i - S$ also contains $p_i$ disjoint cliques of size $c$ for any other choice of a set of vertices $S$, $S \subset V(G_i)$, of size $k$. 

Choose any $S$ with $S \subset V(G_i)$ and $|S|=k$. Let $G''_i = G_i - S$. Note that $|G''_i|=p_ic$. Let us show that $G''_i$ contains $p_i$ disjoint cliques of size $c$. Let $G'' = G-S$. Since $G$ is $k$-FT($pK_c$), $G''$ contains $p$ disjoint cliques of size $c$. Let us use $\mathcal{C}$ to denote this set of cliques. Moreover, let $\mathcal{C}_i$ be the set of elements of $\mathcal{C}$ that intersect $V(A_i)$. Since $W$ is a separator in $G$, the cliques in $\mathcal{C}_i$ are subsets of $V(G_i)$. By the choice of $\mathcal{C}$, they also are subsets of $V(G''_i)$. Since $|S|=k$ and $k<c$, there is $|V(A) \setminus S| > (p_i-1)c$. Thus $|\mathcal{C}_i|=p_i$. This concludes the proof that $G_i$ is $k$-FT($p_iK_c$).

For the final part of the lemma, assume that $W$ is a clique and $G$ is minimum $k$-FT($pK_c$). Choose any $G_i$, like above. So $G_i$ is $k$-FT($p_iK_c$). Towards a contradiction, suppose it is not minimum $k$-FT($p_iK_c$). Consider a graph $\widehat{G_i}$ that is minimum $k$-FT($p_iK_c$). Since $k < c$, by Lemma \ref{lem:ft}, $\widehat{G_i}$ contains a subgraph isomorphic to $K_k$ - let $\widehat{W}$ denote its vertex set. Let $\widehat{G}$ be the graph obtained from $G$ by removing $A_i$ and adding $\widehat{G_i}$, identifying the vertices of $W$ with the ones of $\widehat{W}$. It is easy to check that $\widehat{G}$ is $k$-FT($pK_c$) and $|E(\widehat{G})| < |E|$ - a contradiction with the choice of $G$.
\end{proof}

Let us show that, in a graph $G$ that is $k$-FT($pK_c$), replacing the closed neighborhood of a vertex of degree $c+k-1$ with a copy of $K_k$, gives a graph that is $k$-FT($p'K_c$) with $p'=p-1$.

\begin{lem}\label{lem:contraction}
Let $k$, $p$, and $c$ be integers with $k=1$, $p \geq 1$, and $c \geq 3$. Let $G=(V,E)$ be a $k$-FT($pK_c$) graph with $|V|=pc+k$. Let $x$, $x \in V$, be any vertex with $d_G(x) = c+k-1$. Then the graph $G'$ obtained from $G$ by replacing $N_G[x]$ with a copy of $K_k$, and making every vertex of this $K_k$ adjacent to every vertex in $N_G(N_G[x]))$ (neighborhood of the closed neighborhood of $x$), is $k$-FT($p'K_c$), with $p'=p-1$.
\end{lem}
\begin{proof}
We need to show that, for any $S'$, $S' \subset V(G')$ and $|S'|=k$, the vertices of $G' - S'$ can be partitioned into $p-1$ disjoint cliques of size $c$. Take any such $S'$. Note that $S' \cap N_G[x] = \emptyset$ (since $S' \subset V(G')$).

Note that, by Lemma \ref{lem:Kc+k-1}, $N_G[x]$ is isomorphic to $K_{c+k}$. Let $\widehat{K}$ be the subgraph isomorphic to $K_{k}$ added in $G'$.

Given $S'$ to be removed from $G'$, let us construct the corresponding $S$ to be removed from $G$. If $S' \cap V(\widehat{K}) = \emptyset$, then just let $S=S'$. Otherwise, let $i=|S' \cap V(\widehat{K})|$. By construction $i \leq k$. Let us construct $S$ from $S'$ by replacing the vertices in $S' \cap V(\widehat{K})$ by $i$ arbitrary vertices in $N(x)$. It is possible, since $i \leq k$ and $k\leq c+k-1$.

Since $G$ is $k$-FT($pK_c$), the vertices of $G-S$ can be partitioned into $p$ disjoint cliques of size $c$. Let us use $\mathcal{C}$ to denote this set of cliques. Let us use $C_x$ to denote the clique in $\mathcal{C}$ that contains $x$. Note that $C_x \cap V(G') = \emptyset$ and there are exactly $k-i$ other vertices in $N_G(x) \setminus C_v \setminus S$ that are covered by other cliques in $\mathcal{C}$. So there exists a bijection $b$ between $N_G(x) \setminus C_x \setminus S$ and $V(\widehat{K}) \setminus S'$.

Let us construct $\mathcal{C'}$ by removing $C_x$ from it, and adapting the other cliques by replacing the vertices in $N_G(x) \setminus C_x \setminus S$ by the corresponding vertices in $V(\widehat{K}) \setminus S'$, based on the bijection $b$. Since every vertex in $\widehat{K}$ is adjacent to every vertex in $N_G(N[x])$, the sets thus obtained are cliques indeed. So $\mathcal{C'}$ is a partition of $V(G') \setminus S'$ as needed.
\end{proof}

Note that, in the statement of Lemma \ref{lem:contraction}, if $N_G[x]$ is a separator in $G$, then the copy of $K_k$ it is replaced with in $G'$ is a separator of size $k$ that is a clique, so Lemma \ref{lem:Kkvc} applies.

\begin{figure}[ht] 
\caption{A separator $W=\{v\}$ and $z$ connected components}
\label{separator}
\begin{center}
\includegraphics[width=8cm]{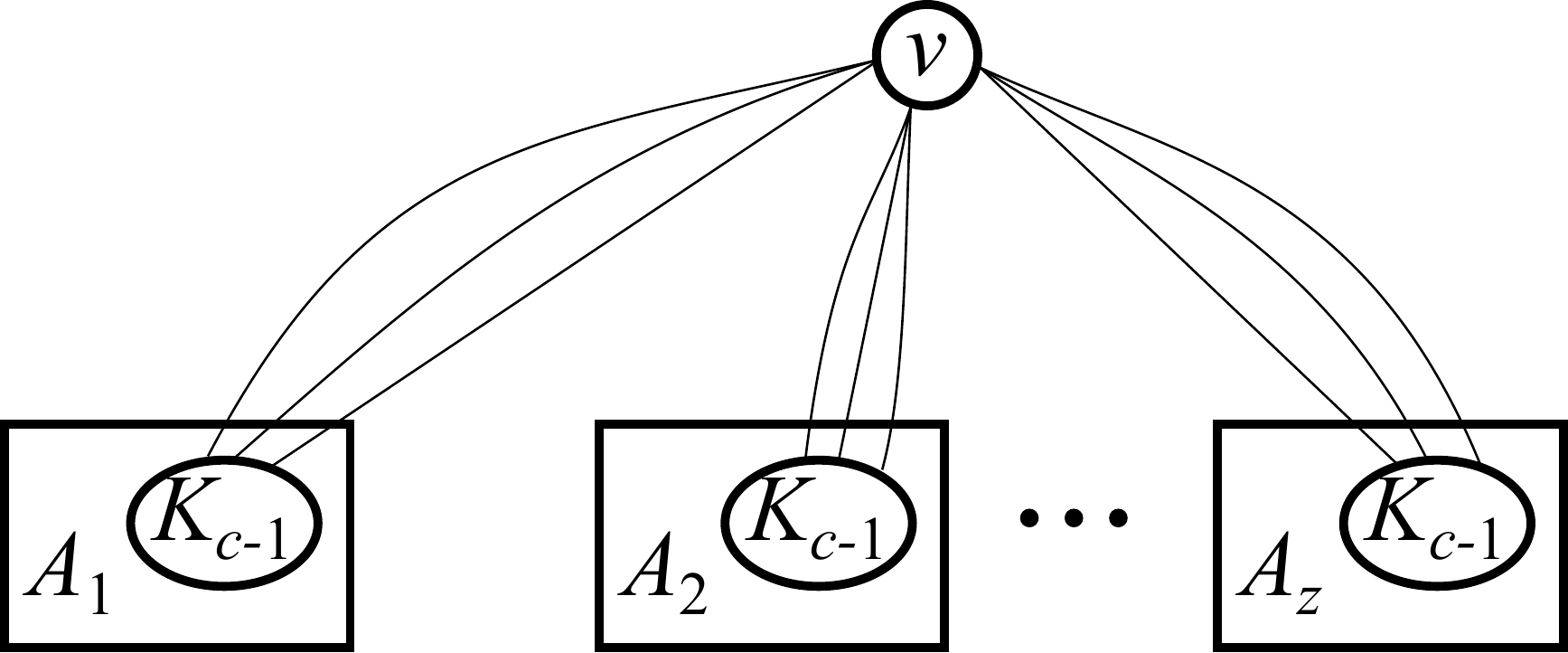}
\end{center}
\end{figure}

\begin{lem}\label{lem:KksKc}
Let $k$, $p$, and $c$ be integers with $k \geq 0$, $p \geq 1$, $c \geq 3$, and $k<c$. Let $G=(V,E)$ be a $k$-FT($pK_c$) graph with $|V|=pc+k$. Let $W$, with $W\subset V$ and $|W|=k$, be a separator in $G$. Let $\{A_i\}_{i=1}^{z}$ be the components of $G'$, $G'=G-W$. Then, for every vertex $x$ in $W$ and every $A_i$, $1 \leq i \leq z$, there exists a $V_{i,x}$, $V_{i,x} \subset V(A_i)$, with $|V_{i,x}|=c-1$, such that $\{x\} \cup V_{i,x}$ is a clique.
\end{lem}

\begin{proof}
Choose any $A_i$, $1 \leq i \leq z$. Let $G_i=G[V(A_i) \cup W]$. By Lemma \ref{lem:Kkvc}, is $k$-FT($p_iK_c$) for some positive integer $p_i$. Choose any $x$ in $W$ and any $x_i$ in $V(A_i)$. Let $W'=W \cup\{x_i\} \setminus\{x\}$. Since $G_i$ is $k$-FT($p_iK_c$), $V(G_i)\setminus W'$ can be partitioned into cliques of size $c$. Since $V(G_i)\setminus W' = V(A_i) \cup \{x\} \setminus \{x_i\}$, there exists a clique $\{x\} \cup V_{i,x}$ of size $c$ as needed.
\end{proof}

In Figure \ref{separator}, we can observe an example of the situation described in Lemma \ref{lem:KksKc} for $k=1$.

Note that, in particular, Lemma \ref{lem:KksKc}, means that in a $k$-FT($pK_c$) graph $G$ of order $pc+k$, with $k \geq 0$, $p \geq 1$, $c \geq 3$, and $k<c$, given a separator $W$ of size $k$, all components of $G-W$ are \textit{full}, i.e., the neighborhood of the vertex set of each of them equals $W$. 

\subsection{Upper bound on the size of minimum $k$-FT($pK_c$) graphs}\label{ss:ub}

Finally, let us give a simple upper bound on the size of a minimum $k$-FT($pK_c$) graph.

\begin{lem}\label{lem:upb}
Let $k$, $p$, and $c$ be integers with $k \geq 0$, $p \geq 1$, $c \geq 3$. Let $G=(V,E)$ be a minimum $k$-FT($pK_c$) graph. Then $|E| \leq \left(\binom{c}{2}+ck\right)p + \binom{k}{2}$.
\end{lem}

\begin{proof}
Let $H$ be the graph obtained by taking $p$ disjoint copies of $K_c$, $1$ copy of $K_k$, and making each vertex of each copy of $K_c$ adjacent to each vertex of $K_k$. It is easy to check that $H$ is $k$-FT($pK_c$) and $|E(H)| = \left(\binom{c}{2}+ck\right)p + \binom{k}{2}$. The conclusion follows.
\end{proof}

Our main result, Theorem \ref{thm:1pc} presented in Section \ref{se:main}, shows that the graphs constructed in the proof of Lemma \ref{lem:upb} are minimum $k$-FT($pK_c$) when $k=1$. We conjecture that they are also minimum when $k>1$.

The construction from the proof of Lemma \ref{lem:upb} can be easily generalized using the notions of tree-decomposition and chordal graph. Let us start by recalling the definition of the former.

\begin{df}[Section 12.3 in \cite{D2017}]
Let $G$ be a graph, $T$ a tree, and let $\mathcal{V} = \{V_t\}_{t \in V(T)}$ be a family of vertex sets $V_t \subseteq V(G)$ indexed by the nodes $t$ of $T$. The pair $(T, \mathcal{V})$ is called a tree-decomposition of $G$ if it satisfies the following three conditions:
\begin{enumerate}
    \item $V(G) = \bigcup_{t \in V(T)} V_t$;
    \item for every edge $e \in E(G)$ there exists $t \in V(T)$ such that $e \subseteq V_t$;
    \item $V_{t_{1}} \cap V_{t_{3}} \subseteq V_{t_{2}}$ whenever $t_2$ lies on the path between $t_1$ and $t_3$ in $T$.
\end{enumerate}
$T$ is the decomposition tree of $(T, \mathcal{V})$, and the elements of $\mathcal{V}$ are the parts of $(T, \mathcal{V})$.
\end{df}

Given two nodes $t_1, t_2$ adjacent in a decomposition tree $T$, $V_{t_1} \cap V_{t_2}$ is the \textit{adhesion set} of $V_{t_1}$ and $V_{t_2}$. The \textit{adhesion} of a tree-decomposition is the maximum size of its adhesion sets.

With the notion of tree-decomposition, we obtain the following characterization of chordal graphs.

\begin{prop}\label{pro:chordal}[Proposition 12.3.6 in \cite{D2017}]
$G$ is chordal if and only if $G$ has a tree-decomposition in which every part is a clique.
\end{prop}

It is easy to check that, without loss of generality, the tree-decomposition of $G$ in Lemma \ref{pro:chordal} can be restricted to have the set of parts equal to the set of maximal cliques of $G$. In this case, the adhesion sets of the tree-decomposition are the minimal separators of $G$ (which are cliques).

\begin{figure}[ht] 
\caption{$2$-FT($5K_3$) graphs}
\label{stabilne}
\begin{center}
\subfigure{\includegraphics[width=7cm]{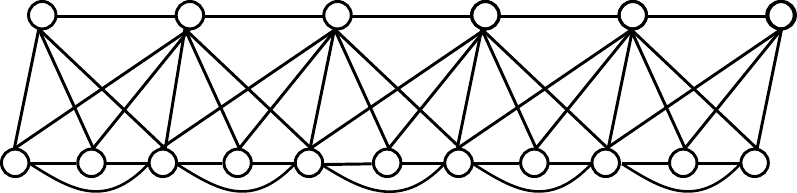}}
\quad
\subfigure{\includegraphics[width=7cm]{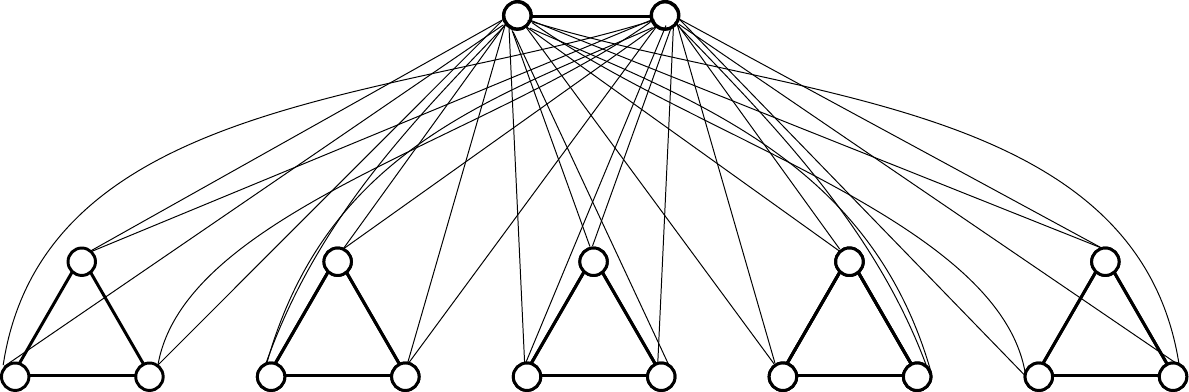}}
\end{center}
\end{figure}

Now we can present a generalization of the construction from the proof of Lemma \ref{lem:upb}. Based on Proposition \ref{pro:chordal}, it is easy to check the validity of the following lemma.

\begin{lem}\label{lem:upb_gen}
Let $k$, $p$, and $c$ be integers with $k \geq 0$, $p \geq 1$, $c \geq 3$. Let $G=(V,E)$ be a chordal graph of order $pc+k$ in which all minimal separators are of size $k$ and all maximal cliques are of size $k+c$. Then $G$ is a $k$-FT($pK_c$) graph with $|E| = \left(\binom{c}{2}+ck\right)p + \binom{k}{2}$.
\end{lem}

In Figure \ref{stabilne} we can see examples of graphs described in Lemma \ref{lem:upb_gen}.

\section{Minimum $1$-FT($pK_c$) graphs}\label{se:main}

In this section we focus on $k$-FT($pK_c$) graphs $G$ with $|V(G)|=pc+k$, for $k=1$, $p \geq 1$, and $c \geq 3$. 

Recall that a block is a maximal connected subgraph without a cutvertex. So, every block is a maximal $2$-connected subgraph, a bridge, or an isolated vertex. Conversely, every such subgraph is a block. Different blocks of a graph $G$ overlap on at most one vertex, which is then a cutvertex of $G$. Every edge of $G$ lies in a unique block, and $G$ is the union of its blocks. Let $A$ be the set of cutvertices of $G$, and $\mathcal{B}$ the set of its blocks. We then have a natural bipartite graph on $A \cup \mathcal{B}$ formed by the edges $\{a,B\}$ with $a \in B$, and the following lemma holds.

\begin{lem}\label{lem:block-graph}[Lemma 3.1.4 in \cite{D2017}]
The block graph of a connected graph is a tree.
\end{lem}

\begin{lem}\label{lem:Kc+1size}
Let $p$, and $c$ be integers with $p \geq 1$ and $c \geq 3$. Let $G=(V,E)$ be a $1$-FT($pK_c$) graph with $|V|=pc+1$. If every block of $G$ is isomorphic to $K_{c+1}$, then $G$ is a chordal graph in which all minimal separators are of size $1$ and all maximal cliques are of size $c+1$. \end{lem}

\begin{proof}
By Lemma \ref{lem:c+k-1ec}, $G$ is $c$ edge-connected. Since $c \geq 3$, $G$ has no isolated vertices nor bridges, so all blocks of $G$ are maximal $2$-connected subgraphs of $G$.

Consider the block graph of $G$. By Lemma \ref{lem:block-graph}, it is a tree. It is easy to check that that this tree gives a tree-decomposition of $G$ in which every part is a clique. So, by Proposition \ref{pro:chordal}, it is a chordal graph.
\end{proof}

\begin{thm}\label{thm:1pc}
Let $p$ and $c$ be integers with $p \geq 1$ and $c \geq 3$. Let $G=(V,E)$ be a $1$-FT($pK_c$) graph with $|V|=pc+1$. If $|E| \leq \binom{c+1}{2} p$, then every block of $G$ is isomorphic to $K_{c+1}$.
\end{thm}

\begin{proof}
Reasoning towards a contradiction, suppose that $G$ is a counterexample of the smallest possible order. I.e., for all $\widehat{p}$, $\widehat{p} < p$, $1$-FT($\widehat{p}K_c$) graphs $\widehat{G}=(\widehat{V},\widehat{E})$ with $|\widehat{V}|=\widehat{p}c+1$ and $|\widehat{E}| \leq \binom{c+1}{2}\widehat{p}$ satisfy the conclusion of the theorem. Notice that $p > 1$, since $K_{c+1}$ is the only $1$-FT($K_c$) graph satisfying the premises - and it also satisfies the conclusion.

By Lemma \ref{lem:c+k-1ec}, $G$ has neither bridges nor isolated vertices. So all blocks of $G$ are maximal 2-connected subgraphs of $G$.

Let us show that $G$ is $2$-connected. Towards a contradiction, suppose it is not. Consider the block graph $BG(G)$ of $G$. Take a block $B$ of $G$ that is a leaf in $BG(G)$., and let $u$ be the cutvertex adjacent to $B$ in $BT(G)$. Thus $u \in V(B)$. 

Let $G[R]$ be the subgraph induced by $R=V \setminus V(B) \cup \{u\}$. By Lemma \ref{lem:Kkvc}, $B$ and $G[R]$ are $1$-FT($p_BK_c$) and $1$-FT($p_RK_c$), respectively, for some positive integers $p_B$ and $p_R$ such that $p_B+p_R=p$. There is $|E(B)| \leq \binom{c+1}{2}p_B$ or $|E(G[R])| \leq \binom{c+1}{2}p_R$, since otherwise we would have $|E(G)| = |E(B)| + |E(G[R])| > \binom{c+1}{2}p_B + \binom{c+1}{2}p_R = \binom{c+1}{2}p$ - a contradiction.

Suppose that $|E(B)| \leq \binom{c+1}{2}p_B$. Then $B$ is isomorphic to $K_{c+1}$, since $B$ is $2$-connected, $G$ is the smallest counterexample, and $B$ is a leaf in $BG(G)$. So we have $|E(B)| = \binom{c+1}{2}p_B$, which implies that $|E(G[R])| \leq \binom{c+1}{2}p_R$. By a similar reasoning, $G[R]$ also satisfies the conclusion of the theorem. Therefore $G$ satisfies the conclusion itself - a contradiction. 

So there is $|E(B)| > \binom{c+1}{2}p_B$ and $|E(G[R])| < \binom{c+1}{2}p_R$. Again, since $G$ is the smallest counterexample, $G[R]$ satisfies the conclusion of the theorem. By Lemma \ref{lem:Kc+1size}, we get that $|E(G[R])| = \binom{c+1}{2}p_R$, and so $|E| > \binom{c+1}{2}p$ - a contradiction. So $G$ is $2$-connected.

Let us show that $G$ contains a vertex $v$ of degree $c$. By Lemma \ref{lem:ft}, there is $d(x) \geq c$ for every vertex $x\in V$. Towards a contradiction, suppose there is $d(x) > c$ for every vertex $x\in V$. So we have $|E| \geq (pc+1)*(c+1)/2  = \binom{c+1}{2}p + \frac{c+1}{2}$. A contradiction with $|E| \leq \binom{c+1}{2}p$.

Let $G' =(V', E') = G \slash N[v]$ be the graph obtained from $G$ by contracting the closed neighborhood of $v$ in $G$. Let $v'$ be the vertex obtained from contracting $N[v]$. By Lemma \ref{lem:contraction}, $G'=(V', E')$ is $1$-FT($p'K_c$) with $p'=p-1$.

The contraction of $N[v]$ eliminated the $\binom{c+1}{2}$ edges of $G[N[v]]$ and there exists a natural injection from the newly added edges incident to $v'$, to the removed edges incident to one vertex in $N[v]$ and another in $V \setminus N[v]$. So there is $|E(G')| \leq \binom{c+1}{2}p'$. Moreover, $G'$ satisfies the conclusion of the theorem, since $G$ is the smallest counterexample.

If $v'$ is a cutvertex, since $G$ is $2$-connected, then $v'$ is the only cutvertex in $G'$ (Any other cutvertex of $G'$ would also be a cutvertex in $G$ - a contradiction.). Otherwise, $G'$ has no cutvertices.   

Suppose that $v'$ is not a cutvertex. So $G'$ is $2$-connected and, not being a counterexample, $G'$ is isomorphic to $K_{c+1}$. So $p'=1$ and $G'$ has $\binom{c+1}{2}$ edges. Moreover, $G$ has $2\binom{c+1}{2}$ edges. Indeed, the contraction diminished the number of edges by at least $\binom{c+1}{2}$ and, satisfying the premises of the theorem, $G$ cannot have more than $2\binom{c+1}{2}$ edges. 

So $G$ is composed of a subgraph isomorphic to $K_{c+1}$ corresponding to $N[v]$, denoted by $\widehat{K_{c+1}}$, a subgraph isomorphic to $K_{c}$ corresponding to the other vertices, denoted by $\widehat{K_{c}}$, and a matching of size $c$ between $\widehat{K_{c+1}}$ and $\widehat{K_{c}}$. Indeed, there must be an edge between each vertex of $\widehat{K_{c}}$ and a vertex of $\widehat{K_{c+1}}$ in order to get a $K_{c+1}$ after the contraction, and the number of such edges must be $c$ to give the total count of $2\binom{c+1}{2}$ in $G$. 

Now let $x$ be any vertex in $\widehat{K_{c}}$. There is $d_{G}(x)=c$ and, by Lemma \ref{lem:Kc+k-1}, $x$ belongs to a subgraph of $G$ isomorphic to $K_{c+1}$ - a contradiction with the above observation on the structure of $G$.

Now suppose that $v'$ is the unique cutvertex in $G'$. So $G'$, not being a counterexample, is isomorphic to $p'$ copies of $K_{c+1}$ sharing the vertex $v'$. So $G'$ has exactly $\binom{c+1}{2}p'$ edges, and the contraction diminished the number of edges by exactly $\binom{c+1}{2}$. On the other hand, $N[v]$ is a separator in $G$, and every component of $G - N[v]$ is a $K_{c}$. Let $\widehat{K_{c}}$ be one of these components. By an argument similar to that of the case when $v'$ is not a cutvertex, $G[V(\widehat{K_{c}}) \cup V(N_G[v])]$ is composed of a $K_{c+1}$, a $K_{c}$, and a matching between them - and $G$ is not $1$-FT($pK_{c}$). A contradiction.
\end{proof}

\section{Conclusions}\label{se:conc}

In this paper we have provided an upper bound on the number of edges in minimum $k$-FT($pK_c$) graphs for $k\geq 1$, $p\geq 1$ and $c\geq 3$ (Lemma \ref{lem:upb_gen}). We have shown that this bound is tight and given a complete characterization of minimum $k$-FT($pK_c$) for $k=1$ (Theorem 1). We conjecture that this bound is also tight for $k > 1$, $p\geq 1$, $c\geq 3$, and $k<c$, and we believe that our techniques could be generalized to prove these cases.

There are also problems left open for $c=2$. For $k=1$, it is easy to check that a cycle $C_{2p+1}$ is minimum $1$-FT($pK_2$). For $k$ even and $2p\leq k-2$, the construction presented in \cite{zhang2012minimum} (based on the Harary graph $H_{m,n}$, which is an $m$-connected graph with $n$ vertices with the smallest number of edges among all such graphs \cite{harary1962maximum}) yields $(2p+k)(k+1)/2$ edges. But other cases are still open. Our techniques from this paper do not apply for $c=2$, so some new ideas are needed to complete the cases with $c=2$.

\bibliography{cycles}

\end{document}